\documentclass[11pt]{amsart}
\usepackage{graphicx}
\usepackage{amsmath, amssymb}
\usepackage{verbatim}
\usepackage{color}
\newtheorem{thm}{Theorem}[section]

\newtheorem{conj.}[thm]{Conjecture}
\newtheorem{lem}[thm]{Lemma}

\theoremstyle{definition}
\newtheorem{defn}[thm]{Definition}
\theoremstyle{remark}

\numberwithin{equation}{section}
\newtheorem{exa}[thm]{Example}

\newcommand{\h}{\mathcal{H}}

\begin{document}

\title[On the duality of c-fusion frames in Hilbert spaces]
{On the duality of c-fusion frames in Hilbert spaces}%
\author[A. Rahimi, Z. Darvishi, B. Daraby]{A. Rahimi, Z. Darvishi, B. Daraby}
\address{Department of Mathematics, University of Maragheh, Maragheh, Iran.}

\email{rahimi@maragheh.ac.ir}
\address{Department of Mathematics, University of Maragheh, Maragheh, Iran.}
\email{darvishi$\_$z@ymail.com}
\address{Department of Mathematics, University of Maragheh, Maragheh, Iran.}
\email{bdaraby@maragheh.ac.ir}

\begin{abstract}
Improving and extending the concept of dual for frames, fusion frames and continuous frames,
      the notion of dual for continuous fusion frames in Hilbert spaces will be studied. It will be shown that generally  the dual of c-fusion frames may not be defined. To overcome this problem, the new concept namely Q-dual for c-fusion frames will be defined and some of its properties will be investigated.
\end{abstract}
\maketitle
\section{Introduction}
Frames arise in many applications in both pure and applied mathematics. Frame theory is nowadays a fundamental research area in mathematics, computer
science and engineering with many exciting and interesting applications in a variety of different
fields. Frames were introduced by Duffin and Scheaffer \cite{duf}
in the context of non-harmonic Fourier series. The importance and  significance of frames for signal processing
has been revealed in the pioneering work by Daubechies, Grossman and
Meyer in 1986,  \cite{mey}. Since then frame theory has quickly become the key approach whenever
redundant, yet stable, representations of data are required.

One goal in frame theory is to construct large frames by fusing
�smaller� frames, and, in fact, this was the original reasoning for introducing fusion
frames by Casazza and Kutyniok, \cite{cas}.

The concept of generalization of frames was proposed by Kaiser \cite{kai}
and independently by Ali, Antoine and Gazeau \cite{ali} to a family indexed by some locally compact space
endowed with a Radon measure. These frames are known as continuous frames. Gabardo
and Han in \cite{gab} called these frames \emph{frames associated with measurable spaces} and Askari-Hemmat, Dehghan and Radjabalipour in \cite{askari} called them \textit{generalized frames} and they are linked to coherent
states in mathematical physics, \cite{ali}. For more studies, the interested reader can also refer to \cite{anto, for,balaz-rahimi,rahimi-najati}.

Combining  fusion and continuous frames was the main motivation for defining the new concept named c-fusion frames \cite{far2,far}.  \par
Throughout this paper, $\h$ is a Hilbert space, $\mathcal{B}(\h)$ the set of all bounded operators on $\h$ and
$\mathbb{H}$ the collection of all closed subspaces of $\h$.
Also, $(X,\mu)$ is a measure space with positive measure $\mu$
and $v:X\to[0,+\infty)$ a measurable mapping such that $v\neq0$ almost every where.
\subsection{Frames in Hilbert spaces.}
\begin{defn}
A countable family $\{h_i\}_{i\in I}$ in a Hilbert space  $\h$ is called   a frame for $\h$
 if there exist  constants $0<A\leq B<+\infty$ such that
 $$A\|h\|^2\leq\sum_{i\in I}|\langle h,h_i\rangle|^2\leq B\|h\|^2,$$ for all $h\in\h$.
 \end{defn}
If the upper bound condition holds, not necessarily the lower bound,
it is called a Bessel sequence.
If $\{h_i\}_{i\in I}$ is a Bessel sequence,
 then the following operators are bounded
$$T_F:l^2(I)\to\h,\hspace{.5cm}T_F\{c_i\}_{i\in I}=\sum_{i\in I}c_ih_i,\quad \{c_i\}_{i\in I}\in l^2(I) $$
$$T_F^*:\h\to l^2(I),\hspace{.5cm}T_F^*h=\{\langle h,h_i\rangle\}_{i\in I},\quad h\in\h$$
and
$$S_F:\h\to\h,\hspace{.5cm}S_Fh=\sum_{i\in I}\langle h,h_i\rangle h_i,\quad h\in\h.$$
These operators are called synthesis operator,
analysis operator and frame operator, respectively.

\hspace{-.2cm}If $\{h_i\}_{i\in I}$ is a frame,
then $S_F$ is bounded, positive and invertible operator and
the following ``reconstruction formula" holds for all $h\in\h$
$$h=\sum_{i\in I}\langle h,S_F^{-1}h_i\rangle h_i=\sum_{i\in I}\langle h,h_i\rangle S_F^{-1}h_i.$$
The family $\{S^{-1}h_i\}_{i\in I}$ is also a frame for $\h$,
called the canonical dual frame of $\{h_i\}_{i\in I}$.
In general, the Bessel sequence $\{k_i\}_{i\in I}$ is
called a dual of the Bessel sequence $\{h_i\}_{i\in I}$,
if the following formula holds
$$h=\sum_{i\in I}\langle h,k_i\rangle h_i$$
 for all $h\in\h$. In this case both  $\{h_i\}_{i\in I}$ and $\{k_i\}_{i\in I}$ are frames for $\h$. For more detail, refer to \cite{ole}.

\subsection{Fusion frames in Hilbert spaces}
\begin{defn}
Let $\h$ be a  Hilbert space and $I$ be a (finite or infinite) countable index set. Assume that $\{W_i\}_{i\in
I}$ is a sequence of closed subspaces of $\h$ and $\{v_i\}_{i\in I}$ is a family
of weights, i.e., $v_i>0$ for all $i\in I$. We say that the
family $\mathcal{W}=\{(W_i,v_i)\}_{i\in I}$ is a  $fusion$ $frame$ or a $frame$ $of$ $subspaces$ with
respect to $\{v_i\}_{i\in I}$ for $\h$ if there
exist constants $0<A\leq B<\infty$ such that
$$A\|x\|^2\leq\sum_{i\in I}v_i^2\|\pi_{W_i}(x)\|^2\leq
B\|x\|^2\quad\forall x\in\h,$$ where $\pi_{W_i}$ denotes the orthogonal
projection onto $W_i,$ for each $i\in I.$ The fusion frame
$\mathcal{W}=\{(W_i,v_i)\}_{i\in I}$ is called  $tight$  if
$A=B$ and $Parseval$ if $A=B=1$.  If all ${v_i}^{,}s$ take the same value $v$, then $\mathcal{W}$ is
called $v$-$uniform$.
Moreover, $\mathcal{W}$ is called an $orthonormal$ $fusion$
 $basis$ for $\h$ if  $\h= \bigoplus_{i\in I} W_i$.
 If $\mathcal{W}=\{(W_i,v_i)\}_{i\in I}$ possesses an
 upper fusion frame bound but not necessarily a lower
 bound, we call it a $Bessel$ $fusion$ $sequence$ with Bessel fusion bound $B$.
 The normalized version of $\mathcal{W}$ is obtained when we choose $v_i=1$ for all $i\in I.$
 Note that, we use this term merely when $\{(W_i,1)\}_{i\in I}$ formes a fusion frame for $\h.$
\end{defn}

We consider the Hilbert space
$$\oplus_{i\in I}W_i=\big\{\{h_i\}_{i\in I}:\hspace{.2cm}h_i\in W_i ,\quad\sum_{i\in I}\|h_i\|^2<+\infty\big\}$$
with inner product $\langle\{h_i\}_{i\in I},\{k_i\}_{i\in I}\rangle=\sum_{i\in I}\langle h_i,k_i\rangle$.
We associate to a Bessel fusion sequence $(W,v)$ the following bounded operators
$$T_{W,v}:\oplus_{i\in I}W_i\to\h,\hspace{.5cm}T_{W,v}\{h_i\}_{i\in I}=\sum_{i\in I}v_ih_i$$
 and
$$T^*_{W,v}:\h\to\oplus_{i\in I}W_i,\hspace{.5cm}T^*_{W,v}h=\{v_i\pi_{W_i}(h)\}_{i\in I}.$$ The operator $T_{W,v}$ called the analysis operator and $T^*_{W,v}$ the synthesis operator.
The operator $S_{W,v}=T_{W,v}T^*_{W,v}$ called the fusion frame operator, it is a  bounded, positive and invertible operator and $A\leq S_{W,v}\leq B.$
\par
In \cite{gav}, G\u{a}vru\c{t}a gave results on the duality of fusion frame in Hilbert spaces.
Heineken and et.al. studied a new concept of dual fusion frame,
entitled, Q-dual of fusion frame in \cite{hei}.
Their definition  as follows:
Assume that $(V,v)$ and $(W,w)$ are fusion frames for $\h$.
If there exists $Q\in B(\mathcal{V},\mathcal{W})$ such that $T_{W,w}QT_{V,v}^*=I_{\h},$
then $(W,w)$ is called Q-dual fusion frame of $(V,v)$.
\hspace{0.5cm}
\subsection{Continuous frames in Hilbert spaces}

\begin{defn}
Let $\h$ be a complex Hilbert space and
$(X,\mu)$ a  measure space with positive measure $\mu$.
A mapping $F:X\rightarrow\h$ is called a
 continuous frame for $\h$ with respect to $(X,\mu)$,
if
\begin{enumerate}
\item for all $h\in\h$, $x\mapsto\langle h,F(x)\rangle$ is a
measurable function on $X$.
\item there exist constants $0<A\leq  B<+\infty$ such that
\begin{equation*}
A\|h\|^2\leq\int_X|\langle h,F(x)\rangle|^2 d\mu(x)\leq B\|h\|^2,\hspace{.5cm}h\in\h;
\end{equation*}
\end{enumerate}
\end{defn}
This mapping is called Bessel if the second inequality holds.
In this case, $B$ is called Bessel constant.
If $F:X\rightarrow\h$ is a Bessel mapping and $\phi\in L^2(X,\mu)$,
then the operators
$$T_F:L^2(X,\mu)\to\h,\hspace{.5cm}T_F(\phi)=\int_X\phi(x)F(x)d\mu(x),\quad \phi\in L^2(X,\mu)$$
$$T_F^*:\h\to L^2(X,\mu),\hspace{.5cm}(T_F^*h)(x)=\langle h,F(x)\rangle,\quad h\in\h, x\in X$$
are well-defined and bounded.
The operator $T_F$ is called the synthesis operator and $T_F^*$ the analysis operator of $F$. The operator $S_F=T_F T^*_F$ is called the continuous frame operator.
Like for discrete frames, this operator is bounded, positive and invertible operator.
It can be shown that for a given continuous frame $F$, the mapping $S_F^{-1}F$ is also a fusion frame, \cite{rahimi-najati}.
It is called canonical dual fusion frame.
Gabardo and Han in \cite{gab} defined  dual frame for a continuous frame as follows:
Let $F$ and $G$ be  continuous frames with respect to $(X,\mu)$ for $\h$.
 The pair $(F,G)$ is called a dual pair,  if $T_GT_F^*=I_{\h}$.

Continuous frames were developed by several authors in different aspects.
One of these  aspects is c-fusion frame, \cite{far2}.
In section 2, we will remind c-fusion frame definition
and in section 3, we state the concept of Q-duality of c-fusion frames in Hilbert spaces.
\section{c-fusion frames}
Let $\mathbb{H}$ be the collection of all closed subspaces of $\h$.
Also, let $F:X\to \mathbb{H}$ and
$L^2(X,\mathcal{H},F)$ be the class of all measurable mappings $f:X\to\h$
such that for each $x\in X$, $f(x)\in F(x)$ and $\int_X\|f(x)\|^2d\mu(x)<\infty$.
It can be verified that $L^2(X,\mathcal{H},F)$ is a Hilbert space with the inner product defined by
$$\langle f,g\rangle=\int_X\langle f(x),g(x)\rangle d\mu(x)$$ for $f,g\in L^2(X,\mathcal{H},F)$.
For brevity, we will denote $L^2(X,\mathcal{H},F)$ by $L^2(X,F)$.
\begin{defn}
Let $F:X\to\mathbb{H}$ be a mapping such that for each $f\in\h$,
the mapping $x\mapsto\pi_{F(x)}(h)$ be measurable.
We say that $(F,v)$ is a c-fusion frame for $\h$
if there exist constants $0<A\leq B<+\infty$ such that
$$A\|h\|^2\leq\int_Xv^2(x)\|\pi_{F(x)}(h)\|^2d\mu(x)\leq B\|h\|^2,\quad h\in\h$$
\end{defn}
\hspace{-.45cm}where $\pi_{F(x)}$ denotes the orthogonal projection of $\h$ on $F(x)$.
It is called a tight c-fusion frame for $\h$ if $A=B$ and Parseval c-fusion frame if $A=B=1$.
If we have only  the upper bound, we call $(F,v)$  a c-Bessel mapping for $\h$.
If $(F,v)$ is a c-Bessel mapping for $\h$,
then we define the synthesis operator
$$T_{F,v}:L^2(X,F)\to\h,\hspace{.5cm}T_{F,v}f=\int_Xv(x)f(x)d\mu(x),  $$ for $f\in L^2(X,F)$.
This operator is a bounded linear operator. Its adjoint
$$T^*_{F,v}:\h\to L^2(X,F),\hspace{.5cm}T^*_{F,v}h=v\pi_F(h),\quad h\in\h$$
is called the analysis operator. The operator $S_{F,v}=T_{F,v}T^*_{F,v}$ is called c-fusion frame operator, it
is positive and self-adjoint operator.
It is easy to verify $$S_{F,v}h=\int_Xv^2(x)\pi_{F(x)}(h)d\mu(x), \quad h\in\h.$$
Furthermore, if $(F,v)$ is a c-fusion frame,
then $S_{F,v}$ is an invertible operator on $\h$
and $$h=\int_Xv^2(x)S_{F,v}^{-1}\pi_{F(x)}(h)d\mu(x),\hspace{.5cm}h\in\h.$$
For more details, refer to  \cite{far}.\par
\section{Q-Duality of c-fusion frames}
Duality of frames is one of the most important aspects of studies on frame theory, not only in theory but in practice \cite{Arefi.4,Fereydooni,Janssen,Laugesen,wave}.
Recall that for two Bessel sequences $\mathcal{F}=\{f_i\}  $ and   $\mathcal{G}=\{g_i\}  $ with synthesis operators $T$ and $U$, ( resp.), the pair $(\mathcal{F},\mathcal{G}   )$ is a dual pair if $TU^*=I$. When we wish to use this definition for c-fusion frames, a problem occurs:

Let $(F,v)$ and $(G,w)$ be two c-fusion frames in $\h$.
Since  $R(T_{F,v}^*)\subseteq L^2(X,F)$ and $D(T_{G,w})=L^2(X,G)$ ( where $R(T_{F,v}^*)$ and $D(T_{G,w})$ denote the range $T_{F,v}^*$ and domain of $T_{G,w}$),
then $T_{G,w}T_{F,v}^*$ is generally not well defined. For overcome this problem, we define a new concept namely Q-dual c-fusion frame.
\begin{defn}
Assume that $(F,v)$ and $(G,w)$ are c-fusion frames for $\h$
with synthesis operators $T_{F,v}$ and $T_{G,w}$, respectively.
If there exists $Q\in \mathcal{B}(L^2(X,F),L^2(X,G))$ such that
$$T_{G,w}QT^*_{F,v}=I_{\h},$$
then $(G,w)$ is called a Q-dual c-fusion frame of $(F,v)$.
\end{defn}
\begin{exa}
We attend to the Hilbert space $\mathbb{R}^2$
with standard base $\{e_1,e_2\}$.
The set $$B_{\mathbb{R}^2}=\{x\in\mathbb{R}^2:~~\|x\|\leq 1\}$$
equipped with Lebesgue measure $\lambda$ forms a measure space.
Suppose $B_1$ and $B_2$ be a partition of $B_{\mathbb{R}^2}$ where $\lambda(B_2)\geq\lambda(B_1)>1$.
For the Hilbert space $\h=\mathbb{R}^2$,
we put $\mathbb{H}=\{W_1,W_2\}$ which $W_1=span\{e_1\}$ and $W_2=span\{e_2\}$.
Define $$F:B_{\mathbb{R}^2}\to\mathbb{H}$$
\[
F(x)=\left\{\begin{array}{ll}
W_1,\hspace{.5cm}x\in B_1,\cr
W_2,\hspace{.5cm}x\in B_2,\cr
\end{array}\right.
\]
and $$v:B_{\mathbb{R}^2}\to[0,+\infty)$$
\[
v(x)=\left\{\begin{array}{ll}
\frac{1}{\sqrt{\lambda(B_1)}},\hspace{.5cm}x\in B_1,\cr
\frac{1}{\sqrt{\lambda(B_2)}},\hspace{.5cm}x\in B_2.\cr
\end{array}\right.
\]
Thus $F$ is weakly measurable and $(F,v)$ is a Parseval c-fusion frame for $\mathbb{R}^2$.
If we define $$G:B_{\mathbb{R}^2}\to\mathbb{H}$$
\[
G(x)=\left\{\begin{array}{ll}
W_2,\hspace{1cm}x\in B_1,\cr
W_1,\hspace{1cm}x\in B_2,\cr
\end{array}\right.
\]
then $(G,v)$ is a Parseval c-fusion frame for $\mathbb{R}^2$.

Now, let $f\in L^2(B_{\mathbb{R}^2},F)$ be arbitrary.
Since for each $x\in X$, $f(x)\in F(x)$, then there exists $\alpha_{f,x},\beta_{f,x}\in\mathbb{R}$
such that
\[
f(x)=\left\{\begin{array}{ll}
\alpha_{f,x}e_1,\hspace{.9cm}x\in B_1,\cr
\beta_{f,x}e_2,\hspace{.95cm}x\in B_2.\cr
\end{array}\right.
\]
Define $Q:L^2(B_{\mathbb{R}^2},F)\to L^2(B_{\mathbb{R}^2},G)$ by
\[
(Qf)(x)=\left\{\begin{array}{ll}
\beta_{f,x}e_2,\hspace{.9cm}x\in B_1,\cr
\alpha_{f,x}e_1,\hspace{.95cm}x\in B_2.\cr
\end{array}\right.
\]
It is easy to verify that $Q\in\mathcal{B}(L^2(B_{\mathbb{R}^2},F),L^2(B_{\mathbb{R}^2},G)).$
Suppose $y=ce_1+de_2\in\mathbb{R}^2$ be arbitrary.
We have
\begin{align*}
T_{G,v}QT_{F,v}^*(y)
& = \int_{B_{\mathbb{R}^2}}v(x)(Qv\pi_Fy)(x)d\lambda(x)\\
& = \int_{B_1}\frac{1}{\lambda(B_1)}de_2d\lambda(x)+\int_{B_2}\frac{1}{\lambda(B_2)}ce_1d\lambda(x)\\
& = de_2+ce_1\\
& = y.
\end{align*}
i.e., $(G,v)$ is a $Q$- dual of $(F,v)$.
\end{exa}
The following theorem characterizes all c-fusion frames in term of a bounded operator $Q$.
\begin{thm}
Let $(F,v)$ be a c-Bessel mapping for $\h$
with synthesis operator $T_{F,v}$.
Then $(F,v)$ is a c-fusion frame for $\h$ if and only if
there exists a c-Bessel mapping $(G,w)$ and a bounded operator $Q\in\mathcal{B}(L^2(X,F),L^2(X,G))$ such that $T_{G,w}QT^*_{F,v}=I_{\h}$.
\end{thm}
\begin{proof}
At the first, suppose $(F,v)$ is a c-fusion frame for $\h$.
For all $h\in\h$
$$h = \int_Xv^2(x)S_{F,v}^{-1}\pi_{F(x)}(h)d\mu(x)
 = \int_Xv(x)S_{F,v}^{-1}v(x)\pi_{F(x)}(h)d\mu(x).$$
Put $G=F$, $w=v$ and $Q=S_{F,v}^{-1}$.
Conversely, suppose that $B_{F,v}$ and $B_{G,w}$ are Bessel constants of c-Bessel mappings $(F,v)$ and $(G,w)$, respectively.
Since $T_{G,w}QT^*_{F,v}=I_{\h}$,
then for each $h\in\h$, we have
$$\langle h,h\rangle = \langle T_{G,w}QT^*_{F,v}(h),h\rangle=\langle QT^*_{F,v}(h),T_{G,w}^*(h)\rangle$$
consequently
\begin{align*}
\|h\|^2
& = |\langle  QT^*_{F,v}(h),T_{G,w}^*(h)\rangle|\\
& \leq \|QT^*_{F,v}(h)\| \|T_{G,w}^*(h)\|\\
& \leq \|Q\| \|T^*_{F,v}(h)\| \|T_{G,w}^*(h)\|\\
& \leq \|Q\|\left(\int_{X}v^2(x)\|\pi_{F(x)}(h)\|^2d\mu(x)\right)^{\frac{1}{2}}\left(\int_{X}w^2(x)\|\pi_{G(x)}(h)\|^2d\mu(x)\right)^{\frac{1}{2}}\\
& \leq \|Q\|\sqrt{B_{G,w}}\|h\|\left(\int_{X}v^2(x)\|\pi_{F(x)}(h)\|^2d\mu(x)\right)^{\frac{1}{2}}.\\
\end{align*}
Thus, for any $h\in\h,$
 $$B_{G,w}^{-1}\|Q\|^{-2}\|h\|^2\leq\int_Xv^2(x)\|\pi_{F(x)}(h)\|^2d\mu(x)$$
i.e., $(F,v)$ is a c-fusion frame for $\h$.
Similarly, $(G,w)$ is a c-fusion frame for $\h$ with lower bound $B_{F,v}^{-1}\|Q\|^{-2}$.
\end{proof}
The following theorem collects some properties of
Q-dual of c-fusion frames that are analogous to corresponding ones for dual frames with similar proofs, \cite{ole}.
\begin{thm}
Let $(F,v)$ and $(G,w)$ be c-Bessel mappings for $\h$
with synthesis operators $T_{F,v}$ and $T_{G,w}$, respectively.
Also, let $Q\in\mathcal{B}(L^2(X,F),L^2(X,G))$.
Then the following conditions are equivalent:
\begin{enumerate}
\item $T_{G,w}QT^*_{F,v}=I_{\h}$;
\item $T_{F,v}Q^*T^*_{G,w}=I_{\h}$;
\item $T^*_{F,v}$ is injective, $T_{G,w}Q$ is surjective and
$$(T^*_{F,v}T_{G,w}Q)^2=T^*_{F,v}T_{G,w}Q;$$
\item $T^*_{G,w}$ is injective, $T_{F,v}Q^*$ is surjective and
$$(T^*_{G,w}T_{F,v}Q^*)^2=T^*_{G,w}T_{F,v}Q^*;$$
\item For all $h,k\in\h$, $\langle h,k\rangle=\langle QT^*_{F,v}(h),T^*_{G,w}(k)\rangle=\langle Q^*T^*_{G,w}(h),T^*_{F,v}(k)\rangle.$
\end{enumerate}
In case any of these equivalent conditions are satisfied,
$(F,v)$ and $(G,w)$ are c-fusion frames,
$(G,w)$ is a Q-dual for $(F,v)$
and $(F,v)$ is a $Q^*$-dual for $(G,w)$.
\end{thm}
The following theorem expresses the uniqueness of the operator $ Q $.
\begin{thm}\label{3.5}
Let $(F,v)$ and $(G,w)$ be two c-Bessel mappings for $\h$
with synthesis operators $T_{F,v}$ and $T_{G,w}$, respectively.
Also, let $T_{F,v}^*$ and $T_{G,w}^*$ be surjective operators.
If there exists $Q\in\mathcal{B}(L^2(X,F),L^2(X,G))$
such that $(G,w)$ is a Q-dual of $(F,v)$,
then the operator $Q$ is unique.
\end{thm}
\begin{proof}
Assume $(G,w)$ is $Q_1$-dual and $Q_2$-dual of $(F,v)$.
Consider $Q=Q_1-Q_2\neq0$.
Then there exists $0\neq f\in L^2(X,F)$ such that $Qf\neq0$.
On the other hand, there exists $h_0\in\h$ such that $T_{F,v}^*h_0=f$,
because $T_{F,v}^*$ is surjective.
Hence $(QT_{F,v}^*)(h_0)\neq0$.
Furthermore $(T_{G,w}QT_{F,v}^*)(h_0)=0$.
Now, we have
$$\langle(T_{G,w}QT_{F,v}^*)(h_0),h\rangle=0,\hspace{.5cm}h\in\h.$$
Thus $0\neq(QT_{F,v}^*)(h_0)\perp range(T_{G,w}^*)$,
which implies that $T_{G,w}^*$ is not surjective.
\end{proof}
Now we want to find a necessary condition on $Q$ such that $(G,v)$ is a Q-dual of $(F,v)$.
We require the following the lemma.
\begin{lem}\cite{naj}\label{3.6}
Let $(F,v)$ be a c-fusion frame for $\h$ with bounds $A_{F,v}$ and $B_{F,v}$.
Then
$$A_{F,v}\dim\h\leq\int_Xv^2(x)\dim F(x)d\mu(x)\leq B_{F,v}\dim\h.$$
Furthermore, if $dim\h<+\infty$, then
$$A_{F,v}\leq\int_Xv^2(x)d\mu(x)\leq B_{F,v}\dim\h.$$
\end{lem}
\begin{thm}
Let $\dim\h<+\infty$.
Let $(F,v)$ and $(G,v)$ be c-Bessel mappings for $\h$ with upper bounds $B_{F,v}$ and $B_{G,v}$, respectively.
Also, let $Q\in\mathcal{B}(L^2(X,F),L^2(X,G))$.
If $(G,v)$ is a Q-dual of $(F,v)$, then
$$\|Q\| \geq \frac{(dim\h)^{-1}}{\sqrt{B_{F,v},B_{G,v}}}.$$
\end{thm}
\begin{proof}
For $h\in\h$
\begin{align*}
\|h\|^2
& \leq \|Q\| \|T^*_{F,v}(h)\| \|T_{G,w}^*(h)\|\\
& \leq \|Q\|\left(\int_{X}v^2(x)\|\pi_{F(x)}(h)\|^2d\mu(x)\right)^{\frac{1}{2}}\left(\int_{X}w^2(x)\|\pi_{G(x)}(h)\|^2d\mu(x)\right)^{\frac{1}{2}}\\
& \leq \|Q\| \|h\|^2 \left(\int_{X}v^2(x)d\mu(x)\right)^{\frac{1}{2}}\left(\int_{X}w^2(x)d\mu(x)\right)^{\frac{1}{2}}.\\
\end{align*}
Therefore, $\|Q\|^{-1}\leq\left(\int_{X}v^2(x)d\mu(x)\right)^{\frac{1}{2}}\left(\int_{X}w^2(x)d\mu(x)\right)^{\frac{1}{2}}$.
 Lemma \ref{3.6} follows  that $$\|Q\|^{-1} \leq \sqrt{B_{F,v},B_{G,v}} \dim\h.$$
\end{proof}
If $(F,v)$ is a c-fusion frame for $\h$ with frame bounds $A_{F,v}$ and $B_{F,v}$,
then $T_{F,v}$ is a bounded mapping and surjective.
Therefore, on $\h$, the operator $T_{F,v}^{\dag}$ is given explicitly by
$T_{F,v}^{\dag}=T_{F,v}^*S_{F,v}^{-1}$ and $$\|T_{F,v}^{\dag}\|\leq\sqrt{\frac{B_{F,v}}{A_{F,v}}}.$$
The following theorem answers to one of the most popular questions on frame theory titled \textit{perturbation of frames} \cite{ole}. The c-fusion version of perturbation theorem  investigated at the following theorem.
\begin{thm}
Suppose $(F,v)$ is a c-fusion frame for $\h$ with frame bounds $A_{F,v}$ and $B_{F,v}$
and $(G,w)$ is a c-Bessel mapping for $\h$.
If there exist $\lambda$ and $\mu$ such that $\lambda+\mu\sqrt{\frac{B_{F,v}}{A_{F,v}}}<1$ and
for all $f\in L^2(X,F)$
\begin{align*}
\|\int_X\left(v(x)f(x)-w(x)Qf(x)\right)d\mu(x)\|
&\leq\lambda\|\int_Xv(x)f(x)d\mu(x)\|\\
& \hspace{.4cm}+\mu\left(\int_X\|f(x)\|^2d\mu(x)\right)^{\frac{1}{2}},
\end{align*}
where $Q\in\mathcal{B}(L^2(X,F),L^2(X,G))$,
then $(G,w)$ is a c-fusion frame.
\end{thm}
\begin{proof}
For each $h\in\h$, we have
\begin{align*}
\|h-T_{G,w}QT_{F,v}^{\dag}(h)\|
& = \|T_{F,v}T_{F,v}^{\dag}(h)-T_{G,w}QT_{F,v}^{\dag}(h)\|\\
& \leq \lambda\|h\|+\mu\|T_{F,v}^{\dag}h\| \\
& \leq \left(\lambda+\mu\sqrt{\frac{B_{F,v}}{A_{F,v}}}\right)\|h\|\\
& < \|h\|.
\end{align*}
Therefore $T_{G,w}QT_{F,v}^{\dag}$ is an invertible operator on $\h$ and
$$\|(T_{G,w}QT_{F,v}^{\dag})^{-1}\|\leq\frac{1}{1-\|I-T_{G,w}QT_{F,v}^{\dag}\|}.$$
Note that for $h\in\h$,
$$h=(T_{G,w}QT_{F,v}^{\dag})(T_{G,w}QT_{F,v}^{\dag})^{-1}(h)$$
and
$$\langle h,h\rangle=\langle(T_{G,w}QT_{F,v}^{\dag})^{-1}(h),(T_{G,w}QT_{F,v}^{\dag})^*(h)\rangle.$$
Thus
\begin{align*}
\|h\|^2
& \leq \|(T_{G,w}QT_{F,v}^{\dag})^{-1}(h)\| \|(T_{G,w}QT_{F,v}^{\dag})^*(h)\|\\
& = \|(T_{G,w}QT_{F,v}^{\dag})^{-1}(h)\| \|(T_{F,v}^{\dag})^*Q^*T_{G,w}^*(h)\|\\
& \leq \|(T_{G,w}QT_{F,v}^{\dag})^{-1}(h)\| \|T_{F,v}^{\dag}\|\|Q\|\|T_{G,w}^*(h)\|\\
& \leq \frac{\sqrt{\frac{B_{F,v}}{A_{F,v}}}\|Q\|}{1-\|I-T_{G,w}QT_{F,v}^{\dag}\|}\|h\|
\left(\int_X w^2(x)\|\pi_{G(x)}(h)\|^2d\mu(x)\right)^{\frac{1}{2}}.
\end{align*}
consequently, $(G,w)$ is a c-fusion frame.
\end{proof}
Now we want to represent the importance of c-fusion frames definition.
It is both necessary and sufficient for us to be able to string together c-fusion frames
for each of the subspace $F(x)$ (with uniformly bounded frame constants)
to get a continuous frame for whole $\h$.
\begin{thm}
For each $x\in X$, let $F(x)\in\mathbb{H}$
and $F_x$ be a continuous frame for $F(x)$
with frame bounds $A_x$ and $B_x$.
Suppose
$$0<A=\inf_{x\in X}A_x\leq\sup_{x\in X}B_x=B<\infty.$$
Also, let $v:X\to[0,+\infty)$ be a measurable mapping such that $v\neq0$ a.e.
The following conditions are equivalent:
\begin{enumerate}
\item $(F,v)$ is a c-fusion frame for $\h$.
\item For each $x\in X$, $v(x)F_x$ is a continuous frame for $\h$.
\end{enumerate}
\end{thm}
\begin{proof}
Since $F_x$ is a continuous frame for $F(x)$
with the frame bounds $A_x$ and $B_x$, we obtain
\begin{align*}
A\int_Xv^2(x)\|\pi_{F(x)}h\|^2d\mu(x)
& \leq \int_Xv^2(x)A_x\|\pi_{F(x)}h\|^2d\mu(x)\\
& \leq \int_X\int_X|\langle h,v(x)F_x(y)\rangle|^2d\mu(y)d\mu(x)\\
&  \leq \int_Xv^2(x)B_x\|\pi_{F(x)}h\|^2d\mu(x)\\
& \leq B\int_Xv^2(x)\|\pi_{F(x)}h\|^2d\mu(x),
\end{align*}
for $h\in\h.$
This shows that provided $v(x)F_x$ is a continuous frame for $\h$ with bounds $A_{v(x)F_x}$ and $B_{v(x)F_x}$,
then $(F,v)$ forms a c-fusion frame for $\h$ with frame bounds $A_{v(x)F_x}B^{-1}$ and $A^{-1}B_{v(x)F_x}$.
Moreover, if $(F,v)$ is a c-fusion frame for $\h$ with frame bounds $A_{F,v}$ and $B_{F,v}$,
the above calculations implies  that
$v(x)F_x$ is a continuous frame for $\h$ with bounds $AA_{F,v}$ and $BB_{F,v}$.
\end{proof}
The following theorem provides a method to obtain dual c-fusion frames in Hilbert spaces.
\begin{thm}
Let $v:X\to[0,+\infty)$ and $w:X\to[0,+\infty)$ be measurable mappings such that $v\neq0$ and $w\neq0$ a.e.
For all $x\in X$, let $F(x),G(x)\in\mathbb{H}$.
Also, suppose $F_x$ and  $G_x$ are continuous frames for $F(x)$ and $G(x)$
with bounds $A_{F_x}$, $B_{F_x}$, $A_{G_x}$ and $B_{G_x}$,
respectively, such that
$$0<A_F=\inf_{x\in X}A_{F_x}\leq\sup_{x\in X}B_{F_x}=B_F<+\infty$$
and
$$0<A_G=\inf_{x\in X}A_{G_x}\leq\sup_{x\in X}B_{G_x}=B_G<+\infty.$$
Define $Q:L^2(X,F)\to L^2(X,G)$ by $$(Qf)(x)=\int_X\langle f(x),F_x(y)\rangle G_x(y)d\mu(y)$$
for each $f\in L^2(X,F)$.
Then the following conditions are equivalent:
\begin{enumerate}
\item for all $x\in X$, $(v(x)F_x,w(x)G_x)$ is a dual pair for $\h$.
\item $(G,w)$ is a Q-dual c-fusion frame $(F,v)$.
\end{enumerate}
\end{thm}
\begin{proof}
Let $x\in X$ and $f\in L^2(X,F)$ be arbitrary.
Since $F_x$ is a continuous frame for $F(x)$, then $\langle f(x),F_x(.)\rangle\in L^2(X)$.
Therefore
 $$\int_X\langle f(x),F_x(y)\rangle G_x(y)d\mu(y)\in G(x),$$
because $G_x$ is a continuous frame for $G(x)$.
The equality $(Qf)(x)=T_{G_x}T_{F_x}^*f(x)$
concludes that
$$\int_X\|(Qf)(x)\|^2\leq B_{G_x}B_{F_x}\int_X\|f(x)\|^2d\mu(x)<+\infty.$$
Hence, $\|Qf\|\leq\sqrt{B_GB_F}\|f\|$ and consequently $Q\in\mathcal{B}(L^2(X,F),L^2(X,G))$.\\
By Theorem \ref{3.5}, it only remains to show the duality condition.
For each $h\in\h$, we have
\begin{align*}
(T_{G,w}QT_{F,v}^*)(h)
& = \int_Xw(x)(QT_{F,v}^*h)(x)d\mu(x)\\
& = \int_X\int_Xw(x)\langle T_{F,v}^*h,F_x(y)\rangle G_x(y)d\mu(y)d\mu(x)\\
& = \int_X\int_Xw(x)v(x)\langle\pi_{F(x)}h,F_x(y)\rangle G_x(y)d\mu(y)d\mu(x)\\
& = \int_X\int_Xw(x)v(x)\langle h,F_x(y)\rangle G_x(y)d\mu(y)d\mu(x)\\
& = \int_X\int_X\langle h,v(x)F_x(y)\rangle w(x)G_x(y)d\mu(y)d\mu(x).
\end{align*}
i.e.,  the condition $(G,w)$ is a Q-dual of c-fusion frame  $(F,v)$,
is equivalent to $(v(x)F_x,w(x)G_x)$ is a dual pair for $\h$ for all $x\in X$.
\end{proof}

\end{document}